\documentclass[12pt]{amsart}

\usepackage[utf8]{inputenc}
\usepackage{amsmath}
\usepackage{amsthm}
\usepackage{amsfonts}
\usepackage{amssymb}
\usepackage{tabu}
\usepackage{latexsym} 
\usepackage[margin=1in]{geometry}
\usepackage{enumerate} 
\usepackage[colorlinks=true]{hyperref}
\usepackage{mathrsfs}
\usepackage{tikz}
\usepackage[parfill]{parskip}
\usetikzlibrary{matrix,arrows,decorations.pathmorphing}
\usepackage{tikz-cd}
\usepackage[all]{xy}
\usepackage{setspace}
\usepackage{tabularx}
\usepackage{xcolor}
\usepackage{framed}

\newcommand{\Ext}{\mathrm{Ext}}
\newcommand{\sExt}{\underline{\mathrm{Ext}}}
\newcommand{\Hom}{\mathrm{Hom}}
\newcommand{\End}{\mathrm{End}}
\newcommand{\sHom}{\underline{\mathrm{Hom}}}
\newcommand{\sEnd}{\underline{\mathrm{End}}}

\newcommand{\ann}{\mathrm{ann}}
\newcommand{\sann}{\underline{\mathrm{ann}}}

\newcommand{\arann}{\mathrm{ARann}}

\newcommand\co{\mathrm{co}}
\newcommand{\m}{\mathfrak{m}}

\newcommand{\module}{\mathrm{mod}}

\newcommand{\MCM}{\mathrm{MCM}}

\newcommand{\add}{\mathrm{add}}

\newcommand{\CC}{\mathcal{C}}

\newcommand{\XX}{\mathcal{X}}

\newcommand{\EE}{\mathcal{E}}
\newcommand{\spec}{\mathsf{Spec}}

\newtheorem{theorem}{Theorem}[section]
\newtheorem*{theorem*}{Theorem}
\newtheorem{lemma}[theorem]{Lemma}
\newtheorem*{lemma*}{Lemma}
\newtheorem{proposition}[theorem]{Proposition}
\newtheorem*{proposition*}{Proposition}
\newtheorem{corollary}[theorem]{Corollary}
\newtheorem*{corollary*}{Corollary}
\newtheorem{question}[theorem]{Question}
\theoremstyle{definition}
\newtheorem{definition}[theorem]{Definition}
\newtheorem*{definition*}{Definition}
\newtheorem{example}[theorem]{Example}
\newtheorem{remark}[theorem]{Remark}

\newtheorem{chunk}[theorem]{}

\newtheorem{introtheorem}{Theorem}

\numberwithin{equation}{theorem}
\def\CC{\mathbb{C}}

\def\e{\boldsymbol{e}}

\def\smd{\operatorname{\mathsf{smd}}}

\def\X{\mathcal{X}}

\DeclareMathOperator{\tr}{\mathrm{tr}}

\newcommand{\Tr}{\mathrm{Tr}}
\newcommand{\nf}{\mathsf{NF}}

\title{Auslander-Reiten annihilators}

\author[\"{O}.~Esentepe]{\"{O}zg\"{u}r Esentepe}
\address{Institut für Mathematik und Wissenschaftliches Rechnen, Universität Graz,
Heinrichstraße 36, 8010 Graz, Austria}
\email{ozgur.esentepe@uni-graz.at}
\urladdr{https://www.sntp.ca}


\begin{document}

\begin{abstract}
    The Auslander-Reiten Conjecture for commutative Noetherian rings predicts that a finitely generated module is projective when certain Ext-modules vanish. But what if those Ext-modules do not vanish? We study the annihilators of these Ext-modules and formulate a generalisation of the Auslander-Reiten Conjecture. We prove this general version for high syzygies of modules over several classes of rings including analytically unramified Arf rings, 2-dimensional local normal domains with rational singularities, Gorenstein isolated singularities of Krull dimension at least 2 and more. We also prove results for the special case of the canonical module of a Cohen-Macaulay local ring. These results both generalise and also provide evidence for a version of the Tachikawa Conjecture that was considered by Dao-Kobayashi-Takahashi.
\end{abstract}

\maketitle

\thispagestyle{empty}

\section{Introduction}

Representation theory of Artin algebras is a source of several long standing homological conjectures and it often helps fueling interesting questions in neighbouring areas. The \textit{Auslander-Reiten Conjecture} is one of the most celebrated examples of this. Motivated by a conjecture of Nakayama \cite{Nakayama} regarding the structure of the minimal injective resolution of a finite dimensional algebra, in \cite{Auslander-Reiten} Auslander and Reiten conjectured that over an Artin algebra, every indecomposable injective module occurs as a direct summand in a minimal injective resolution of the regular module. They proved in the same paper that this \textit{Generalized Nakayama Conjecture} is equivalent to the following: if $\Lambda$ is an Artin algebra and $M$ is a generator such that $\Ext_\Lambda^i(M,M) = 0$ for all $i \geq 1$, then $M$ is projective. This conjecture has made its way to the world of commutative algebra after a couple of decades and the following is now commonly referred to as the \textit{Auslander-Reiten Conjecture} by commutative algebraists.

\textbf{The Auslander-Reiten Conjecture.} Let $R$ be a commutative Noetherian local ring. If $\Ext_R^i(M, M \oplus R)  = $ for all $i \geq 1$, then $M$ is projective (equivalently, free).

It was Auslander, Ding and Solberg who first studied this problem in this context and proved that the conjecture holds for complete intersection rings \cite{Auslander-Ding-Solberg}. In \cite{Araya-Yoshino}, it was proved that the conjecture holds for all modules with finite complete intersection dimension (over any commutative Noetherian local ring). Huneke and Leuschke proved the conjecture for locally excellent Cohen-Macaulay normal rings containing the rational numbers in \cite{Huneke-Leuschke}. And over the last two decades, there have been major interest and improvements (including but not limited to \cite{Araya,Nasseh-Yoshino, Celikbas-Takahashi,Goto-Takahashi,Nasseh-Sather-Wagstaff,Celikbas-Iima-Sadeghi-Takahashi,Sadeghi-Takahashi,Nasseh-Takahashi,Kimura-Otake-Takahashi, Lindo,Avramov-Iyengar-Nasseh, Dey-Ghosh,Ghosh-Takahashi,Kimura-normal}) while the conjecture remains open even for Gorenstein rings. 

During these two decades, several generalizations of the Ext-vanishing condition of the Auslander-Reiten Conjecture were also introduced such as \textit{symmetric Auslander-Reiten condition, generalized Auslander-Reiten condition} or \textit{symmetric Auslander-Reiten condition for modules with constant rank} (see \cite{Huneke-Sega-Vraciu,Wei,Christensen-Holm,Celikbas-Takahashi, Diveris,Dey-Kumashiro-Parangama} and references within). Our goal in this paper is to introduce another generalization of the Auslander-Reiten Conjecture. We will do this by looking at two ideals associated to a finitely generated module over a commutative Noetherian local ring and comparing them.

Our ideals are annihilator ideals. More precisely, they are annihilators of (certain) Ext-modules. Annihilators of Ext modules have been of interest to the commutative algebra community for a very long time and recently there has been several papers connecting them to generation of module categories, derived categories or singularity categories and they have been proven to be useful to understand the singular loci of several classes commutative rings. We refer to \cite{Wang,Iyengar-Takahashi-decomp,Iyengar-Takahashi, Iyengar-Takahashi-jac,Dao-Kobayashi-Takahashi, Dey, Dey-Takahashi, Liu} among many others. Our goal is to connect the theory of these annihilators to the study of the Auslander-Reiten Conjecture.

\begin{chunk}
    \textbf{Stable annihilator ideals.}
For any two $R$-modules $X$ and $Y$, consider the submodule $P(X,Y)$ of $\Hom_R(X,Y)$ consisting of morphisms that factor through a projective module. The quotient is often denoted by $\sHom_R(X,Y)$. The category $\underline{\module}(R)$ whose objects are finitely generated $R$-modules and whose morphism sets are these \textit{stable} morphism sets is called the \textit{stable module category} of $R$. We define the \textit{stable annihilator} of $M$ as the annihilator of the stable endomorphism ring (i.e. the endomorphism ring in the stable module category) of $M$ and we denote it by $\sann_R (M)$. It is immediate from this definition that 
    \begin{align}\label{stable-annihilator-of-a-projective}
        \sann_R(M) = R \iff M \text{ is projective}.
    \end{align}
    We also have the following equalities:
    \begin{align*}
        \sann_R(M) = \ann_R \Ext_R^1(M, \Omega_R M) & = \bigcap_{i > 0 } \bigcap_{N \in \module R} \ann_R \Ext_R^i(M,N)
    \end{align*}
    where $\Omega_R(-)$ is the syzygy operator \cite[Lemma 2.2]{Dao-Kobayashi-Takahashi}. The vanishing locus of this ideal is precisely the non-free locus 
    \begin{align*}
        \nf(M) \colon \{ p \in \spec R \colon M_p \text{ is not free as an } R_p\text{-module}\}.
        \end{align*}
    of $M$.
\end{chunk}

\begin{chunk}
    \textbf{Auslander-Reiten annihilators.} The next ideal of our interest is the \textit{Auslander-Reiten annihilator} of $M$ which we introduce in this paper. We consider the ideal
    \begin{align*}
        \arann_R(M) := \bigcap_{i > 0}  \ann_R \Ext_R^i(M, M \oplus R)
    \end{align*}
    and with this notation, the Auslander-Reiten Conjecture can be restated as 
    \begin{align}\label{ar-annihilator-of-a-projective}
        \arann_R(M) = R \iff M \text{ is projective}.
    \end{align}
    Again, as the annihilators and Ext-modules behave nicely with localisation, the Auslander-Reiten conjecture implies that $\arann_R(M)$ should also define the non-free locus of $M$.
\end{chunk}

\begin{chunk}
    \textbf{Our motivating question.} 
    Comparing \ref{stable-annihilator-of-a-projective} and \ref{ar-annihilator-of-a-projective}, we see that whenever $R$ satisfies the Auslander-Reiten Conjecture, we have $\sann_R(M) = R$ if and only if $\arann_R(M) = R$. Moreover, by the ongoing discussion, we know that the Auslander-Reiten Conjecture holds for a commutative Noetherian local ring $R$ if and only if for any finitely generated $R$-module $M$, there is an equality $\sqrt{\sann_R(M)} = \sqrt{\arann_R(M)}$.

    Motivated by this, we are interested in the following question.
    \begin{question}\label{main-question-1}
        For an arbitrary finitely generated module $M$ over $R$, when do we have an equality
        \begin{align*}
            \sann_R (M) = \arann_R(M)?
        \end{align*}
    \end{question}
    Note that by the above discussion, we know that $\sann_R(M)$ is contained in $\ann_R\Ext_R^i(M, N)$ for any $i > 0$ and any $R$-module $N$. Thus, we always have an inclusion $\sann_R(M) \subseteq \arann_R(M)$. The equality would follow, for instance, by showing that for some $i > 0$, either $\ann_R\Ext_R^i(M,M)$ or $\ann_R\Ext_R^i(M,R)$ is contained in $\sann_R(M)$.

    Let us illustrate two main examples.
    \begin{example}\label{hypersurface-example}
        \textbf{Hypersurface rings.}
        Let $R$ be a hypersurface singularity and $M$ be a maximal Cohen-Macaulay $R$-module. Then, a minimal complete resolution of $M$ is 2-periodic \cite{Eisenbud} and we have isomorphisms
        \begin{align*}
            \sEnd_R(M) \cong \sExt_R^2(M,M) \cong \Ext_R^2(M,M)
        \end{align*}
        where $\sExt_R^i(-,*)$ denotes the \textit{stable Ext} or \textit{Tate cohomology} groups \cite{Buchweitz}. Thus, we have an equality
        \begin{align*}\
            \sann_R (M) = \ann_R\Ext_R^2(M,M)
        \end{align*}
        and consequently, we have $\sann_R(M) = \arann_R(M)$.
    \end{example}

    \begin{example}\label{nearly-gorenstein-example}
        \textbf{Nearly Gorenstein rings.}
        Let $(R, \m)$ be a Cohen-Macaulay local ring with a canonical module $\omega$. Assume that $R$ is \textit{nearly Gorenstein} but not Gorenstein which means that we have an equality $\tr(\omega) = \sann_R(\omega) = \m$ \cite{Herzog-Takayuki-Dumitru}. It is clear from the defining property that a nearly Gorenstein ring is Gorenstein on the punctured spectrum and therefore generically Gorenstein. Hence, in this case $R$ satisfies a particular case of the Auslander-Conjecture, namely the \textit{Tachikawa Conjecture} which states (in our notation) that $\arann_R(\omega) = R$ if and only if $R$ is Gorenstein \cite{Avramov-Buchweitz-Sega}. We assumed that $R$ is not Gorenstein and so $\arann_R(\omega) \neq R$ and we know that $\m = \sann_R(\omega) \subseteq \arann_R(\omega)$. Consequently, we obtain 
        \begin{align*}
            \tr(\omega) = \sann_R(\omega) = \arann_R(\omega) = \m.
        \end{align*}
        
    \end{example}
\end{chunk}

\begin{chunk}\label{introduction-the-category-e}
    \textbf{The category $\EE$.} In the next section, we study some basic properties of the full subcategory 
    \begin{align*}
        \EE(R) = \{ X \in \module R \colon \sann_R(X) = \arann_R(X) \} 
    \end{align*}
    of the module category. We note that if $\EE(R) = \module(R)$, then $R$ satisfies the Auslander-Reiten Conjecture. 
    
    By using standard techniques homological algebra, one can verify the Auslander-Reiten Conjecture for the class of modules with finite projective dimension. Indeed, we have $\Ext_R^{m}(M, R) \neq 0$ whenever the module $M$ has finite projective dimension $m \geq 1$. On the other hand, with the help of the functor $\Hom_R(M,-)$, the natural short exact sequence $ 0 \to \Omega M \to F \to M \to 0$ (where $F$ is a free module and $\Omega M$ is a syzygy of $M$) gives us a long exact sequence which yields isomorphisms $\Ext_R^i(M,M) \cong \Ext_R^i(\Omega M, \Omega M)$ for all $i > 0$ provided that $\Ext_R^i(M,R) = 0$ for all $i > 0$. Hence, if the condition $\Ext_R^i(M, M \oplus R) = 0$ is satisfied for all $i>0$, then the condition $\Ext_R^i(\Omega M, \Omega M \oplus R) = 0$ is also satisfied for all $i > 0$. Therefore, by induction, it is enough to consider the conjecture for (higher) syzygies. 
    
    Consequently, we have that the Auslander-Reiten Conjecture holds for the ring $R$ if we have an inclusion $\Omega_R^n(\module R) \subseteq \EE(R)$ for some $n \geq 0$. Thus, we believe that studying this category (and other related categories) is a useful thing to do. In the next section, we study properties of this category and prove such as being closed under direct sums, direct summands and syzygies. The following can be considered as the main theorem of Section 2 which is a combination of Lemma \ref{main-direct-sum-lemma} and Corollary \ref{smd-of-the-category-E}.
    \begin{introtheorem}
        Let $R$ be a commutative Noetherian local ring.
        \begin{enumerate}
            \item The category is $\EE(R)$ is closed under finite direct sums.
            \item We have $\smd(\EE(R)) = \module R$ where $\smd(\EE(R))$ is the smallest full subcategory that contains $\EE(R)$ and is closed under direct summands.
        \end{enumerate}
    \end{introtheorem}
\end{chunk}
In Section 2, we also give examples of rings $R$ where, for some $n \geq 0$, $\Omega^n(\module R)$ is contained in $\EE(R)$ in Krull dimension 1 (Arf rings), 2 (rational singularities) and 3 (the (2-1)-scroll).

\iffalse
\begin{chunk}
    \textbf{Radical stable annihilators.} In Section 3, we discuss radical stable annihilators. Assuming that $R$ locally satisfies the Auslander-Reiten Conjecture, we prove the following theorem as an affirmative answer to Question \ref{main-question-1}.
    \begin{introtheorem}
        Let $R$ be a commutative Noetherian local ring such that $R_p$ satisfies the Auslander-Reiten Conjecture for every prime ideal $p$ (e.g. $R$ is normal) and $M$ be a finitely generated $R$-module. If $\sann_R(M)$ is radical, then we have an equality $\sann_R(M) = \arann_R(M)$.
    \end{introtheorem}
    As an application, we study conductor ideals of seminormal extensions.
\end{chunk}
\fi

\begin{chunk}
    \textbf{Cohen-Macaulay rings.} In Section 3, we focus our attention to Cohen-Macaulay rings. We start with looking at one-dimensional Cohen-Macaulay local rings and see examples where we observe the inclusions 
    \begin{align*}
    \Omega(\module R) \subseteq \MCM(R) \subseteq \EE(R).
    \end{align*}
    See, for instance, Proposition \ref{far-flung-Gorenstein-nearly} for the case where $R$ is far-flung Gorenstein and nearly Gorenstein. After this, we utilise Auslander-Reiten theory for maximal Cohen-Macaulay modules which are free on the punctured spectrum. The following is the main theorem of Section 3.
    
    \begin{introtheorem}
        Let $R$ be a Gorenstein local ring of Krull dimension at least 2 and assume that $R$ is an isolated singularity. Then, for any maximal Cohen-Macaulay $R$-module $M$, we have $\arann_R(M) = \sann_R(M)$.
    \end{introtheorem}

    We also use similar methods to provide a more general version of a theorem of Dao-Kobayashi-Takahashi, namely \cite[Theorem 4.4]{Dao-Kobayashi-Takahashi}. See Proposition \ref{radical-canonical-trace-proposition}, Corollary \ref{Gorenstein-order-corollary} and Theorem \ref{type-2-theorem}.

    \begin{introtheorem}
        Let $R$ be a Cohen-Macaulay local ring with canonical module $\omega$. Assume that one of the following conditions is satisfied.
        \begin{enumerate}
            \item $R$ is generically Gorenstein and the canonical trace ideal $\tr(\omega)$ is radical.
            \item $R$ is Gorenstein on the punctured spectrum and Henselian and is of Krull dimension at least 2 such that the endomorphism ring $\End_R(R \oplus \omega)$ is a Gorenstein $R$-order.
            \item $R$ is Gorenstein on the punctured spectrum, is of type 2 and is of Krull dimension at least 2.
        \end{enumerate}
        Then, we have an equality $\sann_R(\omega) = \arann_R(\omega)$. 
    \end{introtheorem}
\end{chunk}
 
\section{Preliminaries}
In this section, we are going to recall some background, set up conventions and notation and prove some preliminary results. Throughout this paper, we will always consider commutative Noetherian local rings and finitely generated modules over them. Hence, a module will always refer to a finitely generated module. Also, we will use free and projective interchangeably as we will be dealing with local rings.
\begin{chunk}
    \textbf{Some general definitions and conventions.} 
    Let us recall some definitions and facts that we will frequently use throughout this paper.
\begin{definition}\label{general-definitions}
    Let $R$ be a commutative Noetherian local ring and $M$ be an $R$-module. 
    \begin{enumerate}
        \item We denote by $(-)^*$ the functor $\Hom_R(-,R) \colon \module R \to \module R$. We call $M^*$ the $R$-\textit{dual} of $M$. A module $M$ is called \textit{reflexive} if the canonical map $M \to M^{**}$ is an isomorphism. 
        \item We denote by $D(-)$ the functor $\Hom_R(-,\omega)$ when $R$ is a Cohen-Macaulay local ring with canonical module $\omega$. It is a duality when restricted to maximal Cohen-Macaulay modules and we call $DM$ the \textit{canonical dual} of $M$.
        \item Given a surjection $F \xrightarrow{p} M \to 0$ with $F$ a free module, we call the kernel of $p$ a \textit{syzygy} of $M$ and denote it by $\Omega_R M$. This is uniquely defined up to projective summands. We drop the subscript when it is clear from the context.
        \item Given a free presentation $F \xrightarrow{d} G \to M \to 0$ of the module $M$, the cokernel of $d^*$ is called the \textit{Auslander transpose} of $M$ and we denote it by $\Tr M$. This is also uniquely defined up to projective summands. By definition, we have $M^* = \Omega^2 \Tr M $ up to projective summands. In particular, we have an isomorphism $\Ext_R^i(M^*, -) \cong \Ext_R^{i+2}(\Tr M, -)$ for any $i \geq 1$.
        \item If $\Ext_R^i(\Tr M, R) = 0$ for $i = 1, \ldots, n$, then $M$ is called $n$-\textit{torsionfree}. If $M$ is reflexive and we have $\Ext_R^i(M,R) = \Ext_R^i(M^*,R) = 0$ for all $i \geq 1$, then $M$ is called \textit{totally reflexive} (or sometimes called \textit{of Gorenstein dimension zero}). \cite{Auslander-Bridger}.
    \end{enumerate}
\end{definition}
\end{chunk}

\begin{chunk}
    \textbf{Stable annihilators.} We shall now recall stable annihilators and prove some preliminary results about them.
    \begin{definition}\label{stable-annihilator-definitions}
        Let $R$ be a commutative Noetherian local ring and $M, N$ be $R$-modules.
        \begin{enumerate}
            \item We denote by $P(M,N)$ the set of all morphisms from $M$ to $N$ that factors through a free $R$-module. This is a submodule of $\Hom_R(M,N)$. We denote the corresponding quotient module by $\sHom_R(M,N)$. When $M = N$, we use the notation $\sEnd_R(M)$.
            \item We denote by $\sann_R(M)$ the annihilator of $\sEnd_R(M)$ and call it the \textit{stable annihilator} of $M$. A ring element $r \in R$ belongs to $\sann_R M$ if and only if we have a commutative diagram
            \begin{align*}
                \xymatrix{
                M \ar^{r}[rr] \ar[dr]  && M \\
                &P \ar[ur]&
                }
            \end{align*}
            with $P$ a projective module.
        \end{enumerate}
        \end{definition}

        Let us record a few facts about stable annihilators some of which may be well-known.

\begin{lemma}\label{first-lemma-for-stable-annihilator}
    Let $R$ be a commutative Noetherian local ring and $M$ be an $R$-module.
    \begin{enumerate}
        \item \cite[Lemma 2.3]{Dao-Kobayashi-Takahashi} We have equalities
        \begin{align*}
            \sann_R (M) = \ann_R \Ext_R^1(M, \Omega M) = \bigcap_{i>0}\bigcap_{N \in \module R}  \ann_R \Ext_R^{i}(M,N). 
        \end{align*}
        \item We have an inclusion $\sann_R(M) \subseteq \sann_R(\Omega M)$.
        \item We have an inclusion $\sann_R(M) \subseteq \sann_R(M^*)$. In particular, we have an equality $\sann_R(M) = \sann_R(M^*)$ when $M$ is reflexive.
        \item For any $n \geq 3$, we have an inclusion 
        \begin{align*}
        \sann_R (M) \subseteq \ann_R \Ext_R^n(\Tr M, R).
        \end{align*}
        \item For any module $N$, we have $\sann_R(M \oplus N) = \sann_R(M) \cap \sann_R(N)$.
        \item If $I$ is an ideal of $R$, then $I \subseteq \sann_R(I)$. If $I$ is a regular trace ideal, then there is equality.
    \end{enumerate}
\end{lemma}
\begin{proof}
    The second assertion follows from the first. Indeed, the first assertion says that for any $R$-module $N$ and positive integer $i$, there is an inclusion $\sann_R(M) \subseteq \ann_R\Ext_R^i(M,N)$. Then, the result follows from the fact that we have
    \begin{align*}
        \sann_R(\Omega M) = \ann_R \Ext_R^1(\Omega M, \Omega^2M) = \ann_R \Ext_R^2(M, \Omega^2 M).
    \end{align*}
    For the third assertion note that the following commutative diagram on the left gives rise to the commutative diagram on the right.

    \begin{minipage}{.5\textwidth}
        \begin{align*}
            \xymatrix{
            M \ar^{r}[rr] \ar[dr]_\alpha  && M \\
            &F \ar[ur]_\beta&
            }
        \end{align*}
    \end{minipage}%
    \begin{minipage}{.5\textwidth}
        \begin{align*}
            \xymatrix{
            M^* \ar^{r}[rr] \ar[dr]_{\beta ^*}  && M^* \\
            &F^* \cong F \ar[ur]_{\alpha ^*}&
            }
        \end{align*}
    \end{minipage}

    For the next assertion, we recall that $\Ext_R^n(\Tr M, R) \cong \Ext_R^{n-2}(M^*, R)$. Thus, 
    \begin{align*}
        \sann_R(M) \subseteq \sann_R(M^*) \subseteq \ann_R\Ext_R^{n-2}(M^*, R) = \ann_R \Ext_R^n(\Tr M, R).
    \end{align*}
    The fifth assertion follows from the fact that $\sHom(-,*)$ is additive on both components and $\sHom_R(M,N)$ is a module over $\sEnd_R(M)$.
    
    Finally, for the last assertion, we have the diagram
    \begin{align*}
        \xymatrix{
            I \ar^{r}[rr] \ar@{^{(}->}[dr]_{\iota}  && I \\
            &R \ar[ur]_r&
            }
    \end{align*}
    for every $r \in I$ and the second part follows from the main theorem of \cite{Dey}.
\end{proof}
\end{chunk}

\begin{chunk}
    \textbf{Auslander-Reiten annihilators.} Finally, let us introduce the Auslander-Reiten annihilators. We refer to the discussion in the Introduction for our motivation to consider this ideal.
    \begin{definition}\label{Auslander-Reiten-annihilator-definitions}
        Let $R$ be a commutative Noetherian local ring and $M, N$ be  $R$-modules.
            \begin{enumerate}
                \item We put 
                \begin{align*}
                \e(M,N) := \bigcap_{i > 0} \ann_R \Ext_R^i(M,N).
                \end{align*}
                \item We put $\arann_R(M) := \e(M,M \oplus R)$ and call it the \textit{Auslander-Reiten annihilator} of $M$.
            \end{enumerate}
    \end{definition}
    \begin{proposition}\label{first-proposition-for-auslander-reiten-annihilators}
        Let $R$ be a commutative Noetherian local ring and $M$ an $R$-module.
        \begin{enumerate}
            \item We have an inclusion $\sann_R(M) \subseteq \e(M,N)$ for any $R$-module $N$. In particular, we have $\sann_R(M) \subseteq \arann_R(M)$.
            \item We have an inclusion $\arann_R(M \oplus N) \subseteq \arann_R(M) \cap \arann_R(N)$.
            \item If $M$ is a totally-reflexive module, then there are equalities 
            \begin{align*}
                \arann_R(M) = \e(M,M) = \arann_R(M^*) = \arann_R(\Omega M).
            \end{align*}
            
            \item If there exists a short exact sequence 
            \begin{align*}
            0 \to \Omega M \to X \to M \to 0
            \end{align*}
            such that $\e(M,X) = R$ (for example if $M$ is totally reflexive and $X = R$), then we have $\arann_R(M) = \arann_R(\Omega M)$. 
        \end{enumerate}
    \end{proposition}
    \begin{proof}
        The first assertion follows from the first assertion of Lemma \ref{first-lemma-for-stable-annihilator}. For the second one, we observe
        \begin{align}\label{auslander-reiten-annihilator-direct-sum}
            \arann_R(M \oplus N) = \e(M,N) \cap \arann_R(M) \cap \e(N,M) \cap \arann_R(N). 
        \end{align}
        The third assertion follows from the definition of a totally reflexive module, Lemma \ref{first-lemma-for-stable-annihilator}(3) and the fact that $\Ext_R^i(M, M) \cong \Ext_R^{i}(\Omega_R M, \Omega_R M)$ as discussed in \ref{introduction-the-category-e}.

        For the final assertion, we start by noting that the given short exact sequence induces a long exact sequence 
        \begin{align*}
        \cdots \to \Ext_R^i(M, X) \to \Ext_R^i(M,M) \to \Ext_R^{i+1}(M, \Omega M) \to \Ext_R^{i+1}(M,X) \to \cdots
        \end{align*}
        and by our assumption $\Ext^i(M,X) = 0$ for all $i > 0$. Thus, we get isomorphisms $\Ext_R^i(M,M) \cong \Ext_R^i(\Omega M, \Omega M)$ for all $i > 0$ and the result follows.
    \end{proof}
\end{chunk}

\begin{chunk}
    \textbf{The category $\EE$.} We shall now consider the full subcategory
    \begin{align*}
        \EE = \EE(R) = \{M \in \module R \colon \sann_R(M) = \arann_R(M) \}
    \end{align*}
    of the module category. Let us record some basic properties of $\EE$. We will make use of \ref{auslander-reiten-annihilator-direct-sum} and the inclusions 
    \begin{align}\label{inclusions-direct-sums-stable-annihilator-ar-annihilators}
    \sann_R(M) \cap \sann_R(N) = \sann_R(M \oplus N) \subseteq \arann_R(M \oplus N) \subseteq \arann_R(M) \cap \arann_R(N).
    \end{align}
    We start with an example/lemma.
    \begin{example}
        Let $(R, \m)$ be a commutative Noetherian local ring. If $R$ satisfies the Auslander-Reiten Conjecture, then $\m \in \EE$.
    \end{example}
    \begin{proof}
        If $\m$ is free as an $R$-module (e.g. $R$ is a discrete valuation ring), then $\sann_R(\m) = \arann_R(m) = R$ and so $\m \in \EE$. So, we can assume that $M$ is not free. We have $\m \subseteq \sann_R(\m)$ by Lemma \ref{first-lemma-for-stable-annihilator}(6). We always have $\sann_R(\m) \subseteq \arann_R(\m)$. If this was not an equality, we would have $\arann_R(\m) = R$ but this would be a counterexample for the Auslander-Reiten Conjecture.
    \end{proof}
    We shall record some basic properties of the category $\EE$.
    \begin{lemma}\label{main-direct-sum-lemma}
        Let $R$ be a commutative Noetherian local ring.
        \begin{enumerate}
            \item Then, $\EE$ is closed under finite direct sums.
            \item Let $M_1, \ldots, M_n$ be finitely generated $R$-modules and $a_1, \ldots, a_n$ be positive integers. Then, we have 
            \begin{align*}
                M_1^{\oplus a_1} \oplus \cdots \oplus M_n^{\oplus a_n} \in \EE \iff M_1 \oplus \cdots \oplus M_n \in \EE.
            \end{align*}
            \item Assume that there exists a short exact sequence 
            \begin{align*}
            0 \to \Omega M \to X \to M \to 0
            \end{align*}
            such that $\e(M,X) = R$, then $M \in \EE$ implies $\Omega M \in \EE$. In particular, $\EE$ is closed under syzygies when restricted to totally-reflexive modules.
        \end{enumerate}
    \end{lemma}
    \begin{proof}
        Assuming that $M, N \in \EE$, we have $\arann_R(M) = \sann_R(M) \subseteq \e(M,N)$ and also $\arann_R(N) = \sann_R(N) \subseteq \e(N,M)$. Thus, we can conclude by \ref{auslander-reiten-annihilator-direct-sum}, that we have an equality $\arann_R(M \oplus N) = \sann_R(M \oplus N)$. 
        
        For the second assertion, we note that there are equalities
        \begin{align*}
        \sann_R(M_1^{\oplus a_1} \oplus \cdots \oplus M_n^{\oplus a_n}) &= \sann_R(M_1^{\oplus a_1}) \cap \cdots \cap \sann_R(M_n^{\oplus a_n})\\ &= \sann_R(M_1) \cap \cdots \cap \sann_R(M_n) \\
        & = \sann_R(M_1 \oplus \cdots \oplus M_n) 
        \end{align*}
        and also by additivity of Ext we have the equality 
        \begin{align*}
        \arann_R(M_1^{\oplus a_1} \oplus \cdots \oplus M_n^{\oplus a_n}) = \arann_R(M_1 \oplus \cdots \oplus M_n)
        \end{align*}
        which gives us the desired result.

        The third assertion follows from Proposition \ref{first-proposition-for-auslander-reiten-annihilators}(4) by noting that the existence of such a short exact sequence for a module $M$ implies
        \begin{align*}
            \arann_R(\Omega M) = \arann_R(M) = \sann_R(M) \subseteq \sann_R(\Omega M)
        \end{align*}
        when $M \in \EE$.
    \end{proof}
    \begin{remark}
        In Lemma \ref{main-direct-sum-lemma}(3), we do not need to assume that our modules are totally reflexive. It is enough to assume that they are \textit{semi-Gorenstein projective} in the sense of \cite{ringel-zhang}. However, in several cases, a semi-Gorenstein projective module is indeed totally reflexive. Such cases are given in 
    \end{remark}
While we showed that $\EE$ is closed under direct sums, the next lemma is slightly more stronger.
    \begin{lemma}\label{another-lemma-on-direct-sum}
        Let $R$ be a commutative Noetherian local ring and $M,N$ be two $R$-modules such that $\sann_R(M) \subseteq \sann_R(N)$. If $M \in \EE$, then so is $M \oplus N$. In particular, if $M \in \EE$, then so is $M \oplus M^*$.
    \end{lemma}
    \begin{proof}
        We have $\sann_R(M) = \sann_R(M) \cap \sann_R(N) = \sann_R(M\oplus N)$ by Lemma \ref{first-lemma-for-stable-annihilator}(5) and also $\arann_R(M) \cap \arann_R(N) \subseteq \arann_R(M) = \sann_R(M) $ by assumption. Thus, the inclusions in \ref{inclusions-direct-sums-stable-annihilator-ar-annihilators} become equalities.
    \end{proof}

    \begin{remark}\label{minimal-stable-annihilator}
        Let $R$ be a Cohen-Macaulay local ring with an isolated singularity. Then, there exists a maximal Cohen-Macaulay $R$-module $M_0$ such that for any other maximal Cohen-Macaulay module $M$, we have $\sann_R(M_0) \subseteq \sann_R(M)$ \cite[Theorem 1.2]{Kimura}. Therefore, our lemma tells us that $M_0 \in \EE$ implies $M \oplus M_0 \in \EE$ for any $M \in \MCM(R)$.
    \end{remark}

    \begin{proposition}
        Let $R$ be a Cohen-Macaulay local ring with canonical module $\omega$ and assume that $\omega \in \EE$.
        \begin{enumerate}
            \item For any maximal Cohen-Macaulay module $M$, we have $\arann_R(M \oplus \omega) =  \arann_R(M) \cap \arann_R(\omega)$.
            \item If $M \oplus \omega \in \EE$ for some maximal Cohen-Macaulay $R$-module $M$, then we have an equality $\sann_R(\omega) \cap \sann_R(M) = \sann_R(\omega) \cap \arann_R(M)$. If, moreover, there is an inclusion $\arann_R(M) \subseteq \sann_R(\omega) $, then we have $M \in \EE$.
        \end{enumerate}
        
    \end{proposition}
    \begin{proof}
        For the first assertion, assume $M$ is a maximal Cohen-Macaulay $R$-module. The canonical module is an injective object in the category of maximal Cohen-Macaulay modules. Therefore, we have $\arann_R(M) \subseteq R = \e(M, \omega)$. By assumption, we have $\omega \in \EE$ and thus, we have $\sann_R(\omega) = \arann_R(\omega) \subseteq \e(\omega, M)$. Hence, by \ref{auslander-reiten-annihilator-direct-sum}, we have the equality $\sann_R(M \oplus \omega)= \arann_R(M \oplus \omega)$.

        For the second assertion, we use the inclusions in \ref{inclusions-direct-sums-stable-annihilator-ar-annihilators}. The first inclusion is an equality by our assumption and the second inclusion is an equality by the first assertion. The result follows. 
    \end{proof}
    We continue this section with a discussion on the direct sum of a module and its syzygy. As stated in Lemma \ref{first-lemma-for-stable-annihilator}(1), the stable annihilator of $M$ equals the annihilator of $\Ext_R^1(M, \Omega M)$. This module is a direct summand of $\Ext_R^1(M \oplus \Omega M, M \oplus \Omega M)$ and the next proposition relies on this observation.

    \begin{proposition}\label{direct-sum-of-a-module-and-its-first-syzygy}
        Let $R$ be a commutative Noetherian local ring and $M$ be any finitely generated $R$-module. Then, we have $M \oplus \Omega M \in \EE$.
    \end{proposition}
    \begin{proof}
        We know that we always have $\sann_R(M \oplus \Omega M) \subseteq \arann_R(M \oplus \Omega M)$. So, we should show that $\arann_R(M \oplus \Omega) \subseteq \sann_R(M \oplus \Omega M)$. This follows from 
        \begin{align*}
        \arann_R(M \oplus \Omega M) \subseteq \ann_R \Ext_R^1(M \oplus \Omega M, M \oplus \Omega M) \subseteq \ann_R\Ext_R^1(M , \Omega M) = \sann_R(M)
        \end{align*}
        and the fact that $\sann_R(M) = \sann_R(M \oplus \Omega M) $.
    \end{proof}
    Therefore, when we take the smallest subcategory $\smd(\EE)$ that contains $\EE$ and is closed under direct summands, we get the entire module category.
    \begin{corollary}\label{smd-of-the-category-E}
        For any commutative Noetherian local ring $R$, we have $\smd(\EE) = \module R$.
    \end{corollary}
    We can also strengthen Proposition \ref{direct-sum-of-a-module-and-its-first-syzygy} for modules which are high enough syzygies.
    \begin{proposition}
        Let $R$ ba a commutative Noetherian local ring and $M$ be a finitely generated $R$-module. Assume that for some $r \gg 0$ and an $R$-module $N$, we have $M = \Omega^r N$. Then, for any $n > 0$, we have $M \oplus \Omega^n M \in \EE$.
    \end{proposition}
    \begin{proof}
        We have an ascending chain $\sann_R(N) \subseteq \sann_R(\Omega N) \subseteq \sann_R (\Omega^2 N) \subseteq \cdots$ which stabilizes as $R$ is Noetherian. Therefore, if we choose $r$ large enough, we see that $\sann_R(\Omega^r N) = \sann_R(\Omega^{r+s}N)$ for any $s >0$. This gives us $\sann_R(M) = \sann_R(\Omega^{n-1}M)$. As earlier, we have 
        \begin{align*}
        \arann_R(M \oplus \Omega^n M) \subseteq \ann_R \Ext_R^n(M \oplus \Omega^n M, M \oplus \Omega^n M) \subseteq \ann_R\Ext_R^n(M, \Omega^n M)
        \end{align*}
        and we have 
        \begin{align*}
            \ann_R\Ext_R^n(M, \Omega^n M) = \ann_R\Ext_R^1(\Omega^{n-1}M, \Omega^n M) = \sann_R(\Omega^{n-1}M)
        \end{align*}
        but also we know $\sann_R(\Omega^{n-1}M) = \sann_R(M) = \sann_R(M \oplus \Omega^n M)$ which finishes the proof.
    \end{proof}
    Here is an immediate consequence.
    \begin{corollary}\label{periodic-corollary}
        Let $R$ be a commutative Noetherian local ring and $M$ be a finitely generated $R$-module with a periodic projective resolution. Then, we have $M \in \EE$.
    \end{corollary}
    \begin{remark}\label{remark-about-periodic-modules}
        In fact, we can do better than this in Corollary \ref{periodic-corollary}. Assume that $R$ satisfies the Krull-Remak-Schmidt theorem (\textit{e.g.} $R$ is a Henselian local ring). For a finitely generated $R$-module $M$ with a direct sum decomposition $M = M_1^{\oplus a_1} \oplus \cdots \oplus M_n^{\oplus a_n}$ with $M_1, \ldots M_n$ non-isomorphic indecomposable modules, put $M_{\mathrm{basic}} = M_1 \oplus \cdots \oplus M_n$. By Lemma \ref{main-direct-sum-lemma}(2), we have $M \in \EE$ if and only if $M_{\mathrm{basic}}  \in \EE$. Therefore, instead of assuming that $\Omega^n M \cong M$ for some $n \geq 1$ to conclude $M \in \EE$, we can assume that $(\Omega^n M)_{\mathrm{basic}} \cong M_{\mathrm{basic}}$ up to projective summands.    
    \end{remark}
    \begin{corollary}\label{finite-syzygy-type-corollary}
        Let $R$ be a commutative Noetherian local ring in which Krull-Remak-Schmidt theorem holds and assume that for some $n \geq 0$, we have that $\Omega^n(\module R)$ has only finitely many indecomposable modules up to isomorphism. Then, for some $m \geq n$, we have 
        \begin{align*}
            \Omega^m(\module R) \subseteq \EE.
        \end{align*}
    \end{corollary}
    \begin{proof}
        Assume that there are only finitely many indecomposable modules up to isomorphism in $\Omega^n(\module R)$. Then, in particular, we have finitely many basic modules in $\Omega^n(\module R)$. Hence, for a finitely generated $R$-module $M$, some basic module appears in the sequence $[(\Omega^tM)_{\mathrm{basic}}]_{t \geq 0}$ infinitely often. Then, by Remark \ref{remark-about-periodic-modules} and Corollary \ref{periodic-corollary}, we conclude that some high syzygy, say $\Omega^{m(M)} M$, of $M$ belongs to $\EE$. Then, letting $m = \sup_{M \in \module R} m(M)$, we finish the proof (note that since $\Omega^n{\module R}$ contains finitely many indecomposable modules up to isomorphism, we know that $m$ is finite).  
    \end{proof}
    \begin{example}
        Let us give some examples where the conditions of Corollary \ref{finite-syzygy-type-corollary} are satisfied.
        \begin{enumerate}
            \item In dimension $1$, we have Arf rings. Let $R$ be an Arf ring which has reduced completion. Then, by \cite[Corollary 4.5]{Dao}, we have $\Omega \MCM(R)$ has finite type. Consequently, $\Omega^2 \module R$ has finitely many indecomposable modules up to isomorphism.
            \item In dimension 2, we have rational singularities. Let $R$ be a local normal domain of dimension two. Assume that $R$ is excellent and Henselian and it has an infinite residue field. Then by \cite[Corollary 3.3]{Dao-Iyama-Takahashi-Vial}, $\Omega \MCM(R)$ has finite type if and only if $\spec(R)$ has rational singularities. In this case, we have finitely many indecomposable modules up to isomorphism in $\Omega^3 \module R$.
            \item Let $k$ be an infinite field, $S = k[[x,y,z,u,v]] $ and $I$ be the ideal generated by the $2 \times 2$ minors of the matrix
            \begin{align*}
            A = \begin{bmatrix}
                x & y & u \\ 
                y & z & v
            \end{bmatrix}.
            \end{align*}
            Then, by \cite[16.12]{Yoshino}, the ring $R = S/I$ is a 3-dimensional normal domain of finite Cohen-Macaulay type. Therefore, once again, we have that $\Omega^3 \module R$ contains finitely many indecomposable modules up to isomorphism.
            
        \end{enumerate}
    \end{example}

    We finish this discussion with more remarks on the module $M \oplus \Omega M$. For the rest of this section, assume that $x \in R$ is a nonzerodivisor on $M$. Consider a short exact sequence
    \begin{align*}
        \alpha \colon 0 \to \Omega M \to P \xrightarrow{\epsilon} M \to 0
    \end{align*}
    with $P$ projective and the pullback diagram
    \begin{align*}
        \xymatrix{
                 &    & 0 \ar[d]   & 0 \ar[d] \\
    0 \ar[r] & \Omega M \ar[r] \ar@{=}[d]  & Q \ar[r] \ar[d] & M \ar[r]
    \ar[d]^{x} & 0 \\
    0 \ar[r] & \Omega M \ar[r]  & P \ar[r]^{\varepsilon} \ar[d] & M \ar[r] \ar[d] &
          0\\
                  &    & M/xM \ar@{=}[r] \ar[d]  & M/xM \ar[d] \\
                  &    & 0          & 0
        }
    \end{align*}
    whose middle column tells us that we have a short exact sequence
    \begin{align}\label{definition-of-beta}
        \beta \colon 0 \to \Omega M \to \Omega_R(M/xM) \xrightarrow{\eta} M \to 0
    \end{align}
    in the first row. By construction, as elements of $\Ext_R^1(M,\Omega M)$ we have $\beta = x \alpha$. This implies that if $x$ annihilates $\Ext_R^1(M,\Omega M)$, that is, if $x$ stably annihilates $M$, then the short exact sequence $\beta$ splits. The converse is also true. Therefore, we have $\Omega_R(M/xM) \cong M \oplus \Omega M$ if and only if $x$ stably annihilates $M$. We refer to \cite[Section 2]{Dugas-Leuschke} for more details. Now, the next corollary is Proposition \ref{direct-sum-of-a-module-and-its-first-syzygy} rephrased with this discussion.
    \begin{corollary}
        Let $R$ be a commutative Noetherian local ring and $M$ be a finitely generated $R$-module. If $x \in \sann_R(M)$ is a nonzerodivisor on $M$, then we have  $\Omega_R(M/xM) \in \EE$. In particular, the $R$-module $M/xM$ satisfies the Auslander-Reiten Conjecture.
    \end{corollary}

    \begin{example}
        Assume that $R$ is a Gorenstein local ring and $x$ is a nonzerodivor on $R$ that lies in the \textit{cohomology annihilator} of $R$. This means that $x \in \sann_R(M)$ for any maximal Cohen-Macaulay $R$-module $M$ (this is equivalent to the original definition of the cohomology annihilator ideal \cite[Definition 2.1]{Iyengar-Takahashi} for Gorenstein rings \cite[Lemma 2.3]{Esentepe}). Since $x$ is a nonzerodivisor, it is a nonzerodivisor on any maximal Cohen-Macaulay module $M$. Hence, for any maximal Cohen-Macaulay $R$-module $M$, the $R$-module $M/xM$ satisfies the Auslander-Reiten Conjecture.
    \end{example}

    \begin{remark}
        Similar arguments also hold for regular sequences of longer length by using \cite[Proposition 2.2]{takahashi-thick}.
    \end{remark}
\end{chunk}

\section{Cohen-Macaulay Rings}

In this section, we focus our attention on Cohen-Macaulay rings. 

\begin{chunk}
    \textbf{Dimension 1.} In Example \ref{nearly-gorenstein-example}, we showed that the canonical module $\omega$ belongs to $\EE$ when $R$ is a nearly Gorenstein ring. We did this by arguing that we have 
    \begin{align*}
        \m \subseteq \tr(\omega) = \sann_R(\omega) \subseteq \arann_R(\omega) \neq R.
    \end{align*}
    By the same argument, if $R$ is any commutative Noetherian local ring, $M$ is a finitely generated $R$-module and if we know both $\m \subseteq \sann_R(M)$ and $\arann_R(M) \neq R$, then we can conclude $\sann_R(M) = \arann_R(M)$. Now, we are going to see a class of rings where we can do this for any maximal Cohen-Macaulay module.
    
    \begin{definition}
        Let $R$ be a commutative Noetherian local ring and $Q(R)$ be its total quotient ring. Denote by $\overline{R}$ the integral closure of $R$ inside $Q(R)$. Then, the set 
        \begin{align*}
            \co(R) = \{ r \in \overline{R} \colon r\overline{R} \subseteq R \}
        \end{align*}
    is called the \textit{conductor} of $R$.
    \end{definition}
    It was proved in \cite[Corollary 3.1]{Wang} that if $R$ is a one-dimensional commutative Noetherian complete reduced local ring (in particular, $R$ is Cohen-Macaulay), then for any two maximal Cohen-Macaulay modules $M,N$, we have $\co(R) \subseteq \ann_R\Ext_R^1(M,N)$. 
    \begin{proposition}\label{far-flung-Gorenstein-nearly}
        Assume that $(R, \m)$ is a one-dimensional Noetherian complete reduced local ring such that $\m \subseteq \co(R)$. Then, any maximal Cohen-Macaulay $R$-module belongs to $\EE$.
    \end{proposition}
    \begin{proof}
        Let $M$ be a maximal Cohen-Macaulay $R$-module. By Wang's result mentioned before the statement of the proposition, we have
        \begin{align*}
            \m \subseteq \co(R) \subseteq \ann_R \Ext_R^1(M, \Omega M) = \sann_R(M).
        \end{align*}
        On the other hand, any ring with $\m \subseteq \co(R)$ has minimal multiplicity \cite[Theorem 5.1]{goto-matsuoka-phuong} and such rings satisfy the Auslander-Reiten Conjecture \cite{Dey-Ghosh} and therefore for any non-free $M$, we have $\arann_R(M) \neq R$. This finishes the proof. 
    \end{proof}
    \begin{remark}
        Even without the assumption that $R$ is complete and reduced, we have an inclusion $\co(R) \subseteq \tr(\omega) = \sann_R(\omega)$ by \cite[Proposition A.1]{Herzog-Takayuki-Dumitru}. When $\co(R) = \tr(\omega)$, that is when the canonical trace is as small as it can possibly be, the ring $R$ is called \textit{far-flung-Gorenstein} \cite{herzog-kumashiro-stamate}. Hence, the hypothesis in Proposition \ref{far-flung-Gorenstein-nearly} is saying that $R$ is far-flung Gorenstein and nearly Gorenstein. This implies that $R$ is \textit{almost Gorenstein} \cite[Proposition 1.1]{Huneke-Vraciu} and this was the reason why we were able to use \cite[Theorem 5.1]{goto-matsuoka-phuong} in the proof of Proposition \ref{far-flung-Gorenstein-nearly}.
    \end{remark}
    \begin{example}
        Let $n$ be a positive integer. Consider the numerical semigroup 
        \begin{align*}
            S = \langle n, n+1, \ldots, 2n-1\rangle
        \end{align*}
        and the consider semigroup algebra $R = k[[S]]$ where $k$ is a field. Then, it is easy to see that the conductor of $R$ equals the maximal ideal and therefore we have $\MCM(R) \subseteq \EE(R)$ for such a ring.
    \end{example}
    We are now going to give an application of Remark \ref{minimal-stable-annihilator}.
    \begin{proposition}
        Let $R$ be a Gorenstein local ring of Krull dimension one with reduced completion. Then, we have $\overline{R} \in \EE$ if and only if we have $M \oplus \overline{R} \in \EE$ for any maximal Cohen-Macaulay $R$-module $M$.
    \end{proposition}
    \begin{proof}
        By \cite[Corollary 4.2]{Esentepe}, the conductor ideal coincides with the stable annihilator of the normalisation. Therefore, by Wang's result \cite[Corollary 3.1]{Wang} mentioned above (or more precisely by \cite[Theorem 5.10]{Esentepe} for the reduced completion assumption), the stable annihilator $\sann_R(\overline{R})$
        is contained in the stable annihilator of all maximal Cohen-Macaulay modules. By Remark \ref{minimal-stable-annihilator}, the result follows.
    \end{proof}

\end{chunk}

\begin{chunk}
    \textbf{More on Example \ref{nearly-gorenstein-example}.} Let $R$ be a Cohen-Macaulay local ring with canonical module $\omega$ of Krull dimension $d$. Assume that $R$ is generically Gorenstein (we only need this assumption to ensure that $R$ satisfies the Tachikawa Conjecture \cite{Avramov-Buchweitz-Sega}). As both the annihilators and Ext-modules works well with localisations, as mentioned in the introduction in this case $\arann_R(\omega)$ defines the non-free locus of $\omega$. Hence, we have an equality 
    \begin{align*}
    \sqrt{\arann_R(\omega)} = \sqrt{\sann_R(\omega)} 
    \end{align*}
    of radicals. This gives us the following generalisation of Example \ref{nearly-gorenstein-example}.
    \begin{proposition}\label{radical-canonical-trace-proposition}
    Assume that $R$ is generically Gorenstein and the trace ideal $\tr(\omega)$ is radical. Then, we have $\omega \in \EE$.    
    \end{proposition}
    \begin{proof}
        Since $R$ is generically Gorenstein, we have
        \begin{align*}
            \sann_R(\omega) \subseteq \arann_R(\omega) \subseteq \sqrt{\arann_R(\omega)} = \sqrt{\sann_R(\omega)}. 
            \end{align*}
        When $\sann_R(\omega)$ is radical, all inclusions become equalities.
    \end{proof}
This motivates us to pose the following question.
\begin{question}
    Let $R$ be a Cohen-Macaulay local ring with canonical module $\omega$. When is the canonical trace ideal $\tr(\omega)$ radical? When is it prime?
\end{question}

    We shall now see that we can also replace the generically Gorenstein assumption (or the assumption that $R$ locally satisfies the Tachikawa Conjecture) with some other assumption. Namely, we can assume that $R$ is a Cohen-Macaulay local ring with canonical module $\omega$ and for some positive integer $j$, we have $D \Omega^j \omega$ appears as a syzygy of $\omega$. We start with proving some preliminary lemmas for annihilators of Ext-modules in general. 
    \begin{lemma}\label{general-lemma-for-annihilators}
        Let $0 \to A \to B \to C \to 0$ be a short exact sequence of $R$-modules and $X$ be an $R$-module. Then, we have
        \begin{align*}
            \ann_R\Ext_R^1(A,X) \cdot \ann_R\Ext_R^1(C,X) \subseteq \ann_R\Ext_R^1(B,X).
        \end{align*}
    \end{lemma}
    \begin{proof}
            Let $r\in \ann_R\Ext_R^1(A,X)$ and $s \in \ann_R\Ext_R^1(C,X)$. Applying $\Hom_R(-,X)$ to the short exact sequence gives an exact sequence
            \begin{align*}
                \Ext_R^1(C,X) \xrightarrow{\Phi} \Ext_R^1(B,X) \xrightarrow{\Psi} \Ext_R^1(A,X) 
            \end{align*}
            Assume $\sigma \in \Ext_R^1(B,X)$. Then $\Psi(r \sigma) = r\Psi(\sigma) =0$ and hence $r\sigma \in \ker \Psi $. Hence, There is $\eta \in \Ext_R^1(C,X)$ such that $\Phi(\eta) = r\sigma$. Then,
            \begin{align*}
                sr\sigma = sr\Phi(\eta) = r\Phi(s \eta) = r\Phi(0)=0.
            \end{align*}
    \end{proof}
    \begin{definition}\cite[Definition 5.1]{Dao-Takahashi}
        For two full subcategories $\XX_1, \XX_2$ of $\module R$, define $\XX_1 \ast \XX_2 $ with the following property: A module $M$ belongs to $\XX_1 \ast \XX_2$ if and only if there exists a short exact sequence $0 \to X_1 \to E \to X_2 \to 0$ such that $X_1 \in \add \XX_1, X_2 \in \add \X_2$ and $M$ is a direct summand of $E$. Inductively, put
        \begin{align*}
            |\XX |_r= \begin{cases}
                \add \XX & r = 1 \\
                |\XX|_{r-1} \ast \XX & r \geq 2. 
            \end{cases}
        \end{align*}
        When $\XX$ consists of a single object $X$, we write $|X|_r$ instead.
    \end{definition}
    \begin{lemma}\label{lemma-for-annihilators-from-balls}
        Let $R$ be a commutative Noetherian local ring and $\XX$ be a full subcategory. Assume that $M \in |\XX|_r$. Then, for any $R$-module $Y$, we have
        \begin{align*}
            (\ann_R\Ext_R^1(\XX,Y))^r \subseteq \ann_R\Ext_R^1(M,Y) 
        \end{align*}
        where
        \begin{align*}
            \ann_R\Ext_R^1(\XX,Y) = \bigcap_{X \in \XX} \ann_R\Ext_R^1(X,Y).
        \end{align*}
    \end{lemma}
    \begin{proof}
        For $r=1$, this follows from the fact that
        \begin{align*}
            \ann_R(L_1\oplus L_2) = \ann_R L_1 \cap \ann_R L_2
        \end{align*}
        for any two modules $L_1,L_2$ and the fact that $\Ext_R^1(-,Y)$ distributes over direct sums. For $r \geq 2$, we use Lemma \ref{general-lemma-for-annihilators} inductively.
    \end{proof}

        We know that $\sann_R(\omega) = \ann_R\Ext_R^1(\omega, \Omega \omega)$. By applying the canonical dual, we get the equality $\sann_R(\omega) = \ann_R\Ext_R^1(D\Omega \omega, R)$. This implies that if $D\Omega \omega \in [\XX]_r $, then by Lemma \ref{lemma-for-annihilators-from-balls} we get $\arann_R(\omega)^r \subseteq \sann_R(\omega)$ where $\XX= \{ \Omega^i \omega \colon i \geq 0 \}$. Since we always have $\sann_R(\omega) \subseteq \arann_R(\omega)$, we can conclude that if $D\Omega \omega \in [\XX]_r $ for some $r$, then we have $\sqrt{\sann_R(\omega)} = \sqrt{\arann_R(\omega)}$.
    
        Let us now investivage when $D\Omega \omega \in [\XX]_r $ for some $r$, then. The canonical module is an injective cogenerator in the category $\MCM(R)$ and we can consider an injective coresolution
        \begin{align*}
            0 \to D\Omega\omega \to \omega^{s_0} \to \cdots \to \omega^{s_t} \to U \to 0 
        \end{align*}
        of $D\Omega \omega$ in $\MCM(R)$. Then, by \cite[Lemma 5.8]{Dao-Takahashi}, we have
        \begin{align*}
            D\Omega \omega \in \left| \Omega^{t+1} U \oplus \bigoplus_{i=0}^{t} \Omega^i \omega \right|_{t+2}.
        \end{align*}
        In particular, if $U = \Omega^j \omega$ for some $j \geq 1$, we have $D\Omega \omega \in |\XX|_{t+2}$. This condition is equivalent, by applying $D$, to asking that there is an exact sequence
        \begin{align*}
            0 \to D\Omega^j \omega \to R^{s_t} \to \ldots \to R^{s_0} \to \Omega \omega \to 0
        \end{align*}
        and this is equivalent to asking that $D\Omega^j \omega \cong \Omega^{t+2}\omega$. We record this discussion as a corollary.
        \begin{corollary}
            If there is a $j \geq 1$ such that $D\Omega^j \omega$ appears as a syzygy of the canonical module $\omega$, then $\sann_R(\omega)$ and $\arann_R(\omega)$ agree up to radicals.
        \end{corollary}
        We give an example where the assumption holds.
        \begin{example}\label{Takahashi-example}
            Let $S = \CC[[x_1,\ldots, x_d]]$ be a formal power series ring and $R$ be the second Veronese subring of $S$. Assume $d$ is odd and at least $3$. Then, $R$ is a $d$-dimensional Cohen-Macaulay local ring and it is not Gorenstein \cite[Exercise 3.6.21]{Bruns-Herzog}. As an $R$-module, we have $S \cong R \oplus \omega$. The ring $R$ has an isolated singularity and for every maximal Cohen-Macaulay $R$-module $M$ there is an exact sequence 
            \begin{align*}
                0 \to M \to C^0 \to \cdots \to C^{d-2} \to 0
            \end{align*}
            with $C^i\in\add_R(S)$ for all $i$ \cite[2.2.3]{Iyama}. Since $D \Omega^j \omega$ is a maximal Cohen-Macaulay $R$-module, this gives us the desired result.
        \end{example}
\end{chunk}

\begin{chunk}
    \textbf{Higher dimension.}
We shall now recall some terminology and results from the Auslander-Reiten theory of Cohen-Macaulay local rings.
    
\begin{definition}
    Let $R$ be a commutative Noetherian local ring and $M$ be an $R$-module.
    \begin{enumerate}
        \item We say that $M$ is \textit{locally free on the punctured spectrum} if $M_p$ is a free $R_p$-module for all non-maximal prime ideals $p$.
        \item We say that $R$ is an \textit{isolated singularity} if $R_p$ has finite global dimension for every non-maximal ideal $p$.
        \item Assume $R$ is Cohen-Macaulay with canonical module $\omega$. We say that $R$ is \textit{Gorenstein on the punctured spectrum} if $\omega$ is locally free on the punctured spectrum.
    \end{enumerate}
\end{definition}
We include the following lemmas to keep the paper self-contained.
\begin{lemma}[Lemma 2.10 in \cite{Dao-Kobayashi-Takahashi}]\label{dkt-lemma-210}
    Let $R$ be a d-dimensional Cohen-Macaulay local ring with canonical module $\omega$. Let $M$ be an $R$-module of finite length. Then, we have 
    \begin{align*}
        \ann_R M = \ann_R \Ext_R^d(M, \omega_R).
    \end{align*}
\end{lemma}

\begin{lemma}[Auslander-Reiten Duality. See, for instance, Lemma 3.10 in \cite{Yoshino}]\label{ar-duality}
    Let $R$ be a d-dimensional complete Cohen-Macaulay local ring with canonical module $\omega$ and $M, N$ be two maximal Cohen-Macaulay $R$-modules. If $N$ is locally free on the punctured spectrum, then we have an isomorphism
    \begin{align*}
        \Ext_R^d(\sHom_R(M,N),\omega) \cong \Ext_R^1(N,  D\Omega^d \Tr M).
    \end{align*}
\end{lemma}
Now, combining these two lemmas, we get the following proposition (\textit{c.f. \cite[Proposition 2.11.]{Dao-Kobayashi-Takahashi}})
\begin{proposition}\label{ar-lemma-proposition}
    Let $R$ be a Cohen-Macaulay local ring with canonical module $\omega$ and assume that it is an isolated singularity. Then, 
    \begin{align*}
        \ann_R \sHom_R(M,N) = \ann_R \Ext_R^{d+1}(\Tr M, DN)
    \end{align*}
    for any two maximal Cohen-Macaulay $R$-modules $M,N$.    
\end{proposition}
\begin{proof}
    Since $R$ is an isolated singularity, every maximal Cohen-Macaulay $R$-module is free on the punctured spectrum and for any two maximal Cohen-Macaulay $R$-modules, the module $\sHom_R(M,N)$ is of finite length. Therefore, we can apply Lemma \ref{dkt-lemma-210} to get 
    \begin{align*}
        \ann_R \sHom_R(M,N) = \ann_R \Ext_R^d(\sHom_R(M,N), \omega) 
    \end{align*}
    and we can apply Lemma \ref{ar-duality} to obtain
    \begin{align*}
        \ann_R \sHom_R(M,N) = \ann_R \Ext_R^1(N,  D\Omega^d \Tr M).
    \end{align*}
    By using the canonical duality, we know 
    \begin{align*}
        \Ext_R^1(N,  D\Omega^d \Tr M) \cong \Ext_R^1(\Omega^d \Tr M, DN) \cong \Ext_R^{d+1}(\Tr M, DN).
    \end{align*}
    Combining all, we obtain the desired result.
\end{proof}
The following is a consequence of this proposition for the $R$-dual of the canonical module.
\begin{corollary}\label{the-dual-of-the-canonical-module}
    Let $R$ be a Cohen-Macaulay local ring with canonical module $\omega$. Assume that the Krull dimension of $R$ is at least 2 and that $R$ is Gorenstein on the punctured spectrum. Then, we have $\omega^* \in \EE$. 
\end{corollary}
\begin{proof}
     By Proposition \ref{ar-lemma-proposition}, we have 
    \begin{align*}
        \sann_R(\omega) = \ann_R\Ext_R^{d+1}(\Tr \omega, R) = \ann_R\Ext_R^{d-1}(\omega^*, R) \supseteq \arann_R(\omega^*) \supseteq \sann_R(\omega^*) \supseteq \sann_R(\omega).
    \end{align*}
    Hence, we conclude that $\arann_R(\omega^*) = \sann_R(\omega^*) = \sann_R(\omega)$.
\end{proof}
For the next corollary, let us recall some more terminology.
\begin{definition}
    Let $R$ be a Cohen-Macaulay local ring with canonical module $\omega$ and $\Lambda$ be a module-finite $R$-algebra.
    \begin{enumerate}
        \item We say that $\Lambda$ is an $R$-\textit{order} if $\Lambda$ is maximal Cohen-Macaulay as an $R$-module.
        \item We say that a (right) $\Lambda$-module is \textit{maximal Cohen-Macaulay} if it is maximal Cohen-Macaulay as an $R$-module.
        \item The \textit{canonical module of} $\Lambda$ is the (bi-)module $\omega_\Lambda = \Hom_R(\Lambda, \omega)$.
        \item An order $\Lambda$ is called a \textit{Gorenstein order} if $\omega_\Lambda$ is a projective $\Lambda$-module.
    \end{enumerate}
\end{definition}
The next corollary and the proposition generalize \cite[Theorem 4.4]{Dao-Kobayashi-Takahashi}.
\begin{corollary}\label{Gorenstein-order-corollary}
    Let $R$ be Cohen-Macaulay local ring with canonical module $\omega$ of dimension at least $2$ and assume that it is and Gorenstein on the punctured spectrum and  Henselian. If the endomorphism ring $\End_R(R\oplus \omega)$ is a Gorenstein order, then we have $\omega \in \EE$.
\end{corollary}
\begin{proof}
    By \cite[Corollary 4.6]{Stangle}, with our assumptions, we have $\omega^* \cong \omega$ whenever the endomorphism ring $\End_R(R \oplus \omega)$ is a Gorenstein order. Therefore, we have equalities  $\sann_R(\omega^*) = \sann_R(\omega)$ and $\arann_R(\omega) = \arann_R(\omega^*)$. Therefore Corollary \ref{the-dual-of-the-canonical-module} implies that $\omega \in \EE$.
\end{proof}
\begin{remark}
    When $R$ is a Cohen-Macaulay local domain, the condition $\omega^* \cong \omega$ is equivalent to saying that the order of $[\omega]$ in the divisor class group is $2$. We refer to \cite{Weston} for examples and computations regarding the order of $[\omega]$ in the divisor class group.
\end{remark}
Next, we give a consequence for Cohen-Macaulay local rings of type 2 which are Gorenstein on the punctured spectrum.
\begin{theorem}\label{type-2-theorem}
    Let $R$ be a Cohen-Macaulay local ring of type $2$ and assume that it is Gorenstein on the punctured spectrum. Then, we have $\omega \in \EE$.
\end{theorem}
\begin{proof}
    The key point is that if $R$ is generically Gorenstein and of type $2$, then we have isomorphisms $\Ext_R^i(\omega, R) \cong \Ext_R^{i+1}(\Tr \omega, R)$ for all $i > 0$ by \cite[Lemma 4.1]{Dao-Kobayashi-Takahashi}. Then, as before, we have
    \begin{align*}
        \sann_R(\omega) = \ann_R\Ext_R^{d+1}(\Tr \omega, R) = \ann_R \Ext_R^d(\omega, R) \supseteq \arann_R(\omega) \supseteq \sann_R(\omega) 
    \end{align*}
    which gives us the desired equality.
\end{proof}
Finally, we have the main theorem of this section.
\begin{theorem}
    Let $R$ be a Gorenstein local ring with Krull dimension $d$ at least 2. Then, for any maximal Cohen-Macaulay $R$-module $M$ which is free on the punctured spectrum, we have $M \in \EE$.
\end{theorem}
\begin{proof}
    Let $M$ be a maximal Cohen-Macaulay $R$-module which is free on the punctured spectrum. Then, we have 
    \begin{align*}
        \sann_R(M) = \ann_R \Ext_R^{d+1}(\Tr M, M^*) = \ann_R \Ext_R^{d-1}(M^*,M^*)
    \end{align*}
    where the first equality is by Proposition \ref{ar-lemma-proposition}. We also have 
    \begin{align*}
        \ann_R \Ext_R^{d-1}(M^*,M^*) = \ann_R \Ext_R^{d-1}(M,M) \supseteq \arann_R(M)
    \end{align*}
     where the equality is due to the fact that $(-)^*$ is an exact duality on maximal Cohen-Macaulay modules. By definition, we have $\arann_R(M) \subseteq \ann_R\Ext_R^{d-1}(M,M)$ and combining we get $\arann_R(M) \subseteq \sann_R(M)$ which proves the equality.
\end{proof}
Since every maximal Cohen-Macaulay module is free on the punctured spectrum over an isolated singularity, we have the following corollary.
\begin{corollary}
    Let $R$ be a Gorenstein local ring with Krull dimension $d$ at least $2$. Assume that $R$ has at most an isolated singularity. Then, we have
    \begin{align*}
        \MCM(R) = \Omega^d(\module R) \subseteq \EE.
    \end{align*}
\end{corollary}

\end{chunk}

\section{Acknowledgments}

Many thanks to the anonymous referee for a careful reading of the paper and their valuable suggestions. I would like to thank Eleonore Faber for her detailed reading of a first draft of this paper. Also a thank you to Graham Leuschke and Benjamin Briggs for their comments and Ryo Takahashi for pointing out Example \ref{Takahashi-example}. This research did not receive any specific grant from funding agencies in the public, commercial, or not-for-profit sectors. The author has no competing interests to declare that are relevant to the content of this article.

\bibliographystyle{alpha}
\bibliography{ar-annihilators}

\end{document}